\documentclass[12pt, reqno]{amsart}
\usepackage{amsmath,amsthm,tikz,enumerate,amssymb}
\usepackage[margin=1in]{geometry}
\usepackage[pdfpagelabels,bookmarksopen=true,hidelinks]{hyperref}
\usepackage{todonotes}
\usepackage[nameinlink]{cleveref}
\usetikzlibrary{matrix,arrows}
    \DeclareFontFamily{U}{wncy}{}
    \DeclareFontShape{U}{wncy}{m}{n}{<->wncyr10}{}
    \DeclareSymbolFont{mcy}{U}{wncy}{m}{n}
    \DeclareMathSymbol{\Sh}{\mathord}{mcy}{"58} 
\usepackage[utf8]{inputenc}
\def\C{\mathbf{C}}
\def\Q{\mathbf{Q}}
\def\G{\mathbf{G}}
\def\Z{\mathbf{Z}}
\def\N{\mathbf{N}}
\def\P{\mathbf{P}}
\def\H{\mathrm{H}}
\def\A{\mathbb{A}}
\usepackage[cal=boondoxo, scr=boondoxo]{mathalfa}
\DeclareMathOperator{\Pic}{Pic}
\DeclareMathOperator{\Br}{Br}
\DeclareMathOperator{\Spec}{Spec}
\DeclareMathOperator{\trdeg}{trdeg}

\DeclareMathOperator{\per}{per}
\DeclareMathOperator{\ch}{char}

\DeclareMathOperator{\Aut}{Aut}
\DeclareMathOperator{\Jac}{Jac}
\DeclareMathOperator{\Frac}{Frac}
\DeclareMathOperator{\Mat}{Mat}
\DeclareMathOperator{\Hom}{Hom}
\DeclareMathOperator{\HomS}{\mathcal{H\hspace{-0.3ex}o\hspace{-0.3ex}m}}%\!{\it om}}
\DeclareMathOperator{\m}{{\boldsymbol\mu}}
\numberwithin{equation}{subsection}
\newtheorem{theorem}[subsection]{Theorem}
\newtheorem{lemma}[subsection]{Lemma}
\newtheorem{corollary}[subsection]{Corollary}
\newtheorem{proposition}[subsection]{Proposition}

\theoremstyle{definition}
\newtheorem{definition}[subsection]{Definition}
\newtheorem{question}[subsection]{Question}

\newtheorem{remark}[subsection]{Remark}
\newtheorem{notation}[subsection]{Notation}
\newtheoremstyle{pgstyle} {} {} {} {} {} {.} { } {\textbf{\thmname{#1}\thmnumber{#2}}\thmnote{ (#3)}}
\theoremstyle{pgstyle}
\newtheorem{pg}[subsection]{}
\newtheoremstyle{spgstyle} {} {} {} {} {} {.} { } {\textit{\thmname{#1}\thmnumber{#2}}\thmnote{ (#3)}}
\theoremstyle{spgstyle}
\newtheorem{spg}{}[subsection]
\newcommand{\et}{\operatorname{\acute et}}
\newcommand{\mr}[1]{\mathrm{#1}}
\newcommand{\mc}[1]{\mathcal{#1}}
\newcommand{\ml}[1]{\mathsf{#1}}
\newcommand{\mb}[1]{\mathbf{#1}}
\newcommand{\mf}[1]{\mathfrak{#1}}
\newcommand{\on}[1]{\operatorname{#1}}

\newcommand{\til}{\widetilde}
\begin{document}

\title{Brauertsch fields}
\author{Daniel Krashen}
\author{Max Lieblich}
\author{Minseon Shin}
\begin{abstract} 
    We prove a local-to-global principle for Brauer classes: for any finite collection of non-trivial Brauer classes on a variety over a field of transcendence degree at least 3, there are infinitely many specializations where each class stays non-trivial. This is deduced from a Grothendieck--Lefschetz-type theorem for Brauer groups of certain smooth stacks. This also leads to the notion of a Brauertsch field.
\end{abstract} 
\date{\today}

\maketitle

\tableofcontents

\section{Introduction} \label{sec-01}

In this paper we address the following basic question.

\begin{question}
    Suppose $S$ is a variety over a field $F$ and $\alpha\in\Br(S)$ is a non-zero Brauer class. For how many closed points $s\in S$ is the specialization $\alpha|_{s}$ non-zero?
\end{question}

By analogy with Hilbertian fields, we can codify the non-triviality of specializations of Brauer classes.

\begin{definition}
    A field $F$ is \emph{Brauertsch\/}\footnote{\textbf{Brau}er-Hilb\textbf{ert}-\textbf{sch}} if for any curve $S$ over $F$ and any finite collection of Brauer classes $\alpha_1, \ldots, \alpha_m \in \Br(F(S))$ such that $\per(\alpha_i)$ is invertible in $F$ for all $i$, there are infinitely many closed points $s \in S$ such that each $\alpha_i$ is unramified at $s$ and $\per(\alpha_i|_s) = \per(\alpha_i)$.
\end{definition}

Our main theorem is then the following.

\begin{theorem} \label{main-thm}
    If $F$ has transcendence degree at least $3$ over a perfect field, then $F$ is Brauertsch.
\end{theorem}

\subsection*{Background and existing literature}

Recall the following definition. 
\begin{definition} \cite[VIII]{LANG-DG1962}
    A field $F$ is \emph{Hilbertian} if, for any irreducible polynomial $f(t,x) \in F(t)[x]$, there exist infinitely many constants $a \in F$ such that the specialization $f(a,x) \in F[x]$ is defined and is irreducible in $F[x]$.    
\end{definition}
 The class of Hilbertian fields contains all global fields and the function field $k(t)$ over any field $k$, and every finite extension of a Hilbertian field is Hilbertian (for a survey, see \cite{FRIED-JARDEN-FIELDARITHMETIC} and \cite{SERRE-TIGT2008}). Moreover, the Hilbertian property of a field $F$ has a geometric avatar, namely $F$ is Hilbertian if for any generically finite covering $T \to S$ of curves over $F$, there are infinitely many points $s\in S$ whose preimage in $T$ is a singleton (choose a generically finite map $S \to \P_{F}^1$ and apply the Hilbertian property to the composition $T \to S \to \P_{F}^{1}$).

In \cite{FEINSCHACHERSALTMAN-BHF1992}, Fein, Saltman, Schacher introduced a class of fields satisfying an analogue of the Hilbertian property for the Brauer group. 
\begin{definition}[Fein--Saltman--Schacher]
    A field $F$ is \emph{Brauer-Hilbertian} if, for any positive integer $n$ which is invertible in $F$ and any $n$-torsion Brauer class $\alpha \in \Br(F(t))$, there exist infinitely many constants $a \in F$ such that the specialization $\alpha|_{t=a} \in \Br(F)$ is defined and has period equal to that of $\alpha$, i.e. $\per(\alpha|_{t=a}) = \per(\alpha)$.    
\end{definition}
 By \cite[2.5, 2.6]{FEINSCHACHERSALTMAN-BHF1992}, if $F$ is either a global field or a finitely generated extension of a Hilbertian field $k$ of transcendence degree $\trdeg_{k}(F) \ge 1$, then $F$ is Brauer-Hilbertian (thus, if $S$ is a $d$-fold over $\C$ with function field $F$, then $F$ is Brauer-Hilbertian if and only if $d \ge 2$). 

More generally, one can consider specializations of Brauer classes on an arbitrary extension $K$ of $F$ by choosing a valuation $\ml{v}$ on $K$ such that the specialization of $\alpha$ to the residue field $\kappa_{\ml{v}}$ is defined. In \cite{FEINSCHACHER-SOBCOAFF2000}, Fein and Schacher proved that, if $K$ is a finite extension of $F(t)$ and $\alpha \in \Br(K)$ is a Brauer class in the image of $\Br(F(t)) \to \Br(K)$, then (under certain additional hypotheses on $\alpha$ and $F$) there exist infinitely many $a \in F$ such that the $a$-adic valuation $\ml{v}_{a}$ on $F(t)$ has an extension $\ml{w}$ to $K$ such that the specialization $\alpha|_{\kappa_{\ml{w}}} \in \Br(\kappa_{\ml{w}})$ is defined and has period equal to that of $\alpha$.

\subsection*{Our results}
In this paper we consider a geometric approach to the Brauer-Hilbertian property. It is worth noting that the original results of \cite{FEINSCHACHERSALTMAN-BHF1992} and \cite{FEINSCHACHER-SOBCOAFF2000} cannot be directly obtained from ours, nor ours from theirs. 

\begin{notation}
    For a scheme $S$ and a subgroup $G \subseteq \Br(S)$, let \[ S[G^{-1}] := \{s \in S \;:\; G \to \Br(S) \to \Br(\kappa(s)) \text{ is injective}\} \] be the \emph{nonvanishing locus} of $G$ (viewed as a subset of the underlying topological space of $S$); for a Brauer class $\alpha \in \Br(S)$, we denote $S[\alpha^{-1}] := S[\langle \alpha \rangle^{-1}]$.
\end{notation}
 We are interested in the conditions under which $S[G^{-1}]$ contains specializations of its points. 

\begin{question} \label{0001} If $s \in S[G^{-1}]$ and $s$ is not a closed point, does there exist $s' \in \overline{\{s\}}$ such that $s' \ne s$ and $s' \in S[G^{-1}]$? \end{question}

If $S$ is a surface over $\C$, the answer is ``no''  since the residue field of every non-generic point in a surface is either a $C_{1}$-field or algebraically closed. If $S = \overline{\{s\}}$ and $S$ is a $d$-fold over $\C$ for $d \ge 3$, the theorem of Fein, Saltman, Schacher implies ``yes'' to \Cref{0001} as long as $S$ is birational to $\P_{\C}^{1} \times_{\C} X$ for some $(d-1)$-fold $X$ over $\C$.

We prove the following theorem, which implies \Cref{main-thm} and in particular that \Cref{0001} has an affirmative answer if $S$ is of finite type over $\C$ and $\dim \overline{\{s\}} \ge 4$. 

\begin{notation} For a scheme $S$ and any nonnegative integer $d \in \Z_{\ge 0}$, we denote the set of points of dimension $d$ by $S_{(d)}$ and the set of points of codimension $d$ by $S^{(d)}$. If $S$ is an irreducible scheme, we denote its generic point by $\eta_{S}$. 
\end{notation}

\begin{theorem} \label{0004} Let $k$ be a perfect field, let $F/k$ be a finitely generated extension of transcendence degree $\trdeg_{k}(F) \ge 3$, let $S$ be a finite type $F$-scheme of $\dim S \ge 1$, let $G \subseteq \Br(S)$ be a finite subgroup such that $|G|$ is invertible in $k$. For any $d \ge 1$ and any point $s \in S[G^{-1}] \cap S_{(d)}$, the set $S[G^{-1}] \cap S_{(d-1)} \cap \overline{\{s\}}$ is infinite. \end{theorem}

\begin{corollary} Let $F$ be as in \Cref{0004}, and suppose $S$ is a smooth curve over $F$. If $\alpha \in \Br(S)$ is a nontrivial Brauer class, there exist infinitely many closed points $s \in S$ such that the period of $\alpha|_{s}\in\Br(\kappa(s))$ equals the period of $\alpha$. That is, the restriction map $\Br(S)\to\Br(\kappa(s))$ is injective on the subgroup generated by $\alpha$.
\end{corollary}

To prove \Cref{0004}, we construct a model $X$ of $S$ which is smooth and projective over $k$ and over which $\alpha$ is defined (away from its ramification divisor $D \subset X$). Using that $\alpha$ becomes unramified over the root stack $\mc{X} = \sqrt[\ell]{(X,D)}$ associated to $D$ where $\ell = \per\alpha$, we reduce to the task of lifting an $\alpha$-twisted line bundle from a smooth ample divisor of $\mc{X}$ to one over $\mc{X}$ itself. For this, we prove a Grothendieck--Lefschetz theorem for the Picard group and Brauer group of Deligne-Mumford stacks (see \Cref{0005}). In \Cref{sec-03}, we investigate the problem of constructing Brauer classes $\alpha \in \Br(S)$ whose nonvanishing locus $S[\alpha^{-1}]$ avoids a prescribed set of points $T \subseteq S$, showing that this is always possible if $T$ is a singleton (\Cref{0012}) or if we allow a localization (\Cref{0011}).

\subsection*{Local-to-global consequences}
As an application of our results, in \Cref{sec-04} we prove two kinds of local-to-global principles. In general, given a field $L$ and a collection of overfields $\Omega$ of $L$, we define \[ \Sh_{\Omega}\Br(L) := \ker\left[ \Br(L) \to \prod_{L' \in \Omega} \Br(L') \right] \] and we say that \emph{local-to-global holds with respect to $\Omega$} if $\Sh_{\Omega}\Br(L) = 0$. In practice, one is particularly interested in the case where $\Omega$ consists of a collection of completions of $L$ with respect to a geometric collection of discrete valuations. To this effect, if we are given an integral normal $F$-scheme $S$, we set $\Omega_S$ to be the set of completions of the function field $\kappa(\eta_{S})$ with respect to discrete valuations corresponding to prime divisors of $S$.

\begin{corollary} \label{0022} Let $F$ be as in \Cref{0004}, let $S$ be a smooth $F$-scheme of $\dim S \ge 1$, and let $K := \kappa(\eta_{S})$ be the function field of $S$. Then local-to-global holds with respect to $\Omega_S$, i.e. the natural map
\[ \Br(K) \to \prod_{s \in S^{(1)}} \Br(K_{s}) \] is injective, where we denote $K_{s} := \Frac(\mc{O}_{S,s}^{\wedge})$ for all $s \in S^{(1)}$. \end{corollary}

We note that \Cref{0022} is a particularly strong version of the local-to-global principle, where such results often require that $\alpha$ vanishes at the completions of $K$ at \emph{all} codimension $1$ points on a proper model (e.g. the case of $\mathbb{P}^{1}$ over a $C_{1}$-field, or the case of a semiglobal field, where one also needs all completions on a proper model over the valuation ring).

The second application is a local-to-global principle for genus $1$ curves over function fields of fourfolds, which we prove in \Cref{sec-04}.

\begin{corollary} \label{0015} Let $F,S,K$ be as in \Cref{0022}, and let $C \to \Spec K$ be a genus 1 curve. Suppose that for all codimension 1 points $s \in S^{(1)}$, the base change $C \times_{\Spec K} \Spec K_{s}$ admits a $K_{s}$-point. Then $C$ admits a $K$-point. \end{corollary}

\subsection*{Open questions}

It is tempting to ask the following question, to which we do not at the moment know the answer.

\begin{question}\label{ques:more hilb}
    Given $n>1$, we say that a field $F$ is \emph{$n$-Brauertsch\/} if for any curve $S$ over $F$ and any finite collection $\alpha_1,\dotsc,\alpha_m\in\H^n(S,\G_{m})$ of classes with order invertible in $F$, there are infinitely many closed points $s\in S$ such that the specialization $\alpha_i|_s \in \H^{n}(\Spec \kappa(s),\G_{m})$ has the same order as $\alpha_i$ for all $i$. We say that $F$ is \emph{$n$-Hilbertian\/} if the above condition is only assumed to hold for an open subset $U\subset\P^1_{F}$, one cohomology class, and closed points $s\in U$ with residue field $F$.
    \begin{enumerate}
        \item Is a field of transcendence degree at least $n+1$ over a perfect field $n$-Hilbertian (resp. $n$-Brauertsch)?
        \item More generally, which fields are $n$-Hilbertian (resp. $n$-Brauertsch)?
        \item How are the $n$-Hilbertian and $n$-Brauertsch conditions related?
    \end{enumerate}
\end{question}

\begin{remark}
    If we let $n=1$ in Question \ref{ques:more hilb}, the coefficient sheaf $\G_m$ is not the right choice (by Hilbert's Theorem 90), and one must take finite coefficients (for example, $\m_t$) for the property to naturally compare with the classical Hilbertian property. One could then ask whether the notion of Hilbertian is equivalent to being $1$-Hilbertian (resp. $1$-Brauertsch) in this sense.

    Starting with $n=2$, the inclusion $\Q/\Z\subset\G_m$ induces an isomorphism on cohomology. Moreover, the map $\Z/t\Z=\frac{1}{t}\Z/\Z\to\Q/\Z$ induces a surjection on $t$-torsion cohomology. Thus, there is no distinction (for the purposes of the question) between cohomology of order $t$ with coefficients in $\m_t$, $\Q/\Z$, or $\G_m$.
\end{remark}

\begin{remark}
    The results in this paper, those of \cite{FEINSCHACHERSALTMAN-BHF1992}, and the classical theory of Hilbertian fields imply that the the answer to Question \ref{ques:more hilb}(1) is ``yes'' for $n\leq 2$. What we call ``$2$-Hilbertian'' is Fein, Saltman, and Schacher's notion of ``Brauer--Hilbertian'' in \cite{FEINSCHACHERSALTMAN-BHF1992}.
    
    The whole picture is likely to be more complicated for $n>2$. Indeed, the methods of this paper use the fact that restriction to the generic point is injective for low degree cohomology groups with finite coefficients. This fails in higher degree, rendering an approach to this question based upon purity and the Lefschetz hyperplane theorem for root stacks quite a bit more subtle. That is, there is no reasonable Lefschetz theorem for unramified cohomology (in the sense of Colliot-Thélène).
\end{remark}

\begin{pg}[Acknowledgments] We thank Brian Conrad, Giovanni Inchiostro, Julia Hartmann, Sándor Kovács, and Masahiro Nakahara for helpful conversations. \end{pg}

\section{A Grothendieck--Lefschetz theorem for the Brauer group} \label{sec-02}

Grothendieck's incarnation of the Lefschetz hyperplane theorem for Picard groups (as in \cite{HARTSHORNE-ASOAV1970}) relies on the basic deformation theory of invertible sheaves to lift sheaves off of a divisor to the completion of the ambient variety along that divisor, followed by the algebraization of a certain formal matrix that describes that lifted sheaf. As we explain here, a very similar argument works when the ambient space is a tame, smooth, separated Deligne--Mumford stack with projective coarse moduli space. In \Cref{sec-025}, we show how this applies to specializations in the Brauer group.

\begin{proposition}[Grothendieck--Lefschetz] \label{0005} Let $k$ be a field, let $\mc{X}$ be a tame, smooth, separated Deligne-Mumford stack over $k$ with coarse moduli space $\pi : \mc{X} \to X$. Let $Y \subset X$ be a smooth ample divisor with ideal sheaf $I \subset \mc{O}_{X}$. Assume that \begin{enumerate}[(i)] \item $X$ is a smooth projective $k$-scheme of dimension $\dim X \ge 4$, \item $\pi$ is flat, and \item $\H^{i}(X,I^{n}) = 0$ for $n \ge 1$ and $0 \le i \le 3$. \end{enumerate} Let $Y_{n} := \Spec_{X} \mc{O}_{X}/I^{n+1} \subset X$ be the $n$th thickening of $Y$ in $X$ and set $\mc{Y}_{n}:= \mc{X} \times_{X} Y_{n}$ for all $n \in \N$. \begin{enumerate} \item For any finite locally free $\mc{O}_{\mc{X}}$-module $\mc{E}$ and any open subset $U \subseteq X$ containing $Y$, the map \begin{equation} \label{0024-eqn-04} \Gamma(\mc{X}|_{U},\mc{E}) \to \varprojlim_{n \in \N} \Gamma(\mc{Y}_{n},\mc{E}|_{\mc{Y}_{n}}) \end{equation} is an isomorphism. \item For any algebraic stack $\mc{S}$, let $\on{Vect}_{r}(\mc{S})$ denote the category of finite locally free $\mc{O}_{\mc{S}}$-modules of rank $r$. For any open subset $U \subseteq X$ containing $Y$, the functor \begin{equation} \label{0024-eqn-06} \on{Vect}_{r}(\mc{X}|_{U}) \to \varprojlim_{n \in \N} \on{Vect}_{r}(\mc{Y}_{n}) \end{equation} is fully faithful. The functor \begin{equation} \label{0024-eqn-05} \varinjlim_{Y \subset U \subset X} \on{Vect}_{r}(\mc{X}|_{U}) \to \varprojlim_{n \in \N} \on{Vect}_{r}(\mc{Y}_{n}) \end{equation} is an equivalence of categories. \item The restriction \begin{equation} \label{0005-eqn02} \xi_{\mc{X},i} : \H^{i}(\mc{X},\G_{m}) \to \H^{i}(\mc{X}|_{Y},\G_{m}) \end{equation} is an isomorphism for $i=0,1$. \item The restriction \[ \xi_{\mc{X},2}[\ell] : \H^{2}(\mc{X},\G_{m})[\ell] \to \H^{2}(\mc{X}|_{Y},\G_{m})[\ell] \] is injective for any positive integer $\ell$ which is invertible in $k$.\end{enumerate} 
\end{proposition}

\begin{proof}

By Kresch-Vistoli \cite[Theorem 1]{KV2004}, there exists a smooth projective $k$-scheme $X^{0}$ admitting a finite flat surjective morphism $X^{0} \to \mc{X}$. For $p \ge 0$, let $X^{p} := X^{0} \times_{\mc{X}} \dotsb \times_{\mc{X}} X^{0}$ denote its $(p+1)$-fold fiber product. Since $X^{0} \to \mc{X}$ is finite, it is in particular representable by schemes, so $X^{p}$ is a scheme for all $p \ge 0$. Since $\pi : \mc{X} \to X$ is flat, every $X^{p}$ is finite flat over $X$, hence is Cohen-Macaulay of dimension $\dim X^{p} = \dim X$ by \cite[00R5]{SP}.

Let us form the following cartesian diagram: \begin{equation} \label{0024-eqn-01} \begin{tikzpicture}[>=stealth, baseline=(current bounding box.center)] 
\matrix[matrix of math nodes,row sep=1.7em, column sep=1.7em, text height=1.7ex, text depth=0.5ex] { 
|[name=11]| Y_{0}^{2} & |[name=12]| Y_{1}^{2} & |[name=13]| Y_{2}^{2} & |[name=14]| \dotsb & |[name=15]| X^{2} \\ 
|[name=21]| Y_{0}^{1} & |[name=22]| Y_{1}^{1} & |[name=23]| Y_{2}^{1} & |[name=24]| \dotsb & |[name=25]| X^{1} \\ 
|[name=31]| Y_{0}^{0} & |[name=32]| Y_{1}^{0} & |[name=33]| Y_{2}^{0} & |[name=34]| \dotsb & |[name=35]| X^{0} \\ 
|[name=41]| \mc{Y}_{0} & |[name=42]| \mc{Y}_{1} & |[name=43]| \mc{Y}_{2} & |[name=44]| \dotsb & |[name=45]| \mc{X} \\ 
}; 
\draw[gray,->]
(11) edge (12) (12) edge (13) (13) edge (14) (14) edge (15)
(21) edge (22) (22) edge (23) (23) edge (24) (24) edge (25)
(31) edge (32) (32) edge (33) (33) edge (34) (34) edge (35)
(41) edge (42) (42) edge (43) (43) edge (44) (44) edge (45)
(31) edge (41) (32) edge (42) (33) edge (43) (35) edge (45);
\draw[gray,->,transform canvas={xshift= 3pt}](11) edge (21) (12) edge (22) (13) edge (23) (15) edge (25);
\draw[gray,->,transform canvas={xshift= 0pt}](11) edge (21) (12) edge (22) (13) edge (23) (15) edge (25);
\draw[gray,->,transform canvas={xshift=-3pt}](11) edge (21) (12) edge (22) (13) edge (23) (15) edge (25);
\draw[gray,->,transform canvas={xshift=1.5pt}](21) edge (31) (22) edge (32) (23) edge (33) (25) edge (35);
\draw[gray,->,transform canvas={xshift=-1.5pt}](21) edge (31) (22) edge (32) (23) edge (33) (25) edge (35); \end{tikzpicture} \end{equation}

(1): Let $\mc{E}$ be a finite locally free $\mc{O}_{\mc{X}}$-module. For any open subset $U \subseteq X$ containing $Y$, we have a commutative diagram \begin{equation*} \begin{tikzpicture}[>=stealth, baseline=(current bounding box.center)] 
\matrix[matrix of math nodes,row sep=1.7em, column sep=1.7em, text height=1.7ex, text depth=0.5ex] { 
|[name=21]| \varprojlim_{n \in \N}\Gamma(Y_{n}^{1},\mc{E}|_{Y_{n}^{1}}) & |[name=22]| \Gamma(X^{1}|_{U},\mc{E}|_{X^{1}|_{U}}) \\ 
|[name=31]| \varprojlim_{n \in \N}\Gamma(Y_{n}^{0},\mc{E}|_{Y_{n}^{0}}) & |[name=32]| \Gamma(X^{0}|_{U},\mc{E}|_{X^{0}|_{U}}) \\ 
|[name=41]| \varprojlim_{n \in \N}\Gamma(\mc{Y}_{n},\mc{E}|_{\mc{Y}_{n}}) & |[name=42]| \Gamma(\mc{X}|_{U},\mc{E}|_{\mc{X}|_{U}}) \\ 
}; 
\draw[->,font=\scriptsize]
(22) edge node[above=-1pt] {$\phi^{1}$} (21)
(32) edge node[above=-1pt] {$\phi^{0}$} (31)
(42) edge node[above=-1pt] {\labelcref{0024-eqn-04}} (41)
(41) edge (31) (42) edge (32);
\draw[->,font=\scriptsize,transform canvas={xshift= 1.5pt}](31) edge (21) (32) edge (22);
\draw[->,font=\scriptsize,transform canvas={xshift=-1.5pt}](31) edge (21) (32) edge (22); \end{tikzpicture} \end{equation*} where each column is an equalizer diagram. Since $X^{0}$ and $X^{1}$ are Cohen-Macaulay of dimension $\ge 2$, by \cite[0EL2]{SP} we have that $\phi^{0}$ and $\phi^{1}$ are isomorphisms, hence \labelcref{0024-eqn-04} is an isomorphism.

(2): Let $\mc{E}_{1},\mc{E}_{2} \in \on{Vect}_{r}(\mc{X}|_{U})$ be finite locally free $\mc{O}_{\mc{X}|_{U}}$-modules of rank $r$. Then (1) applied to $\HomS_{\mc{O}_{\mc{X}|_{U}}}(\mc{E}_{1},\mc{E}_{2})$ implies that the map \[ \Hom_{\mc{O}_{\mc{X}|_{U}}}(\mc{E}_{1},\mc{E}_{2}) \to \varprojlim_{n \in \N} \Hom_{\mc{Y}_{n}}(\mc{E}_{1}|_{\mc{Y}_{n}},\mc{E}_{2}|_{\mc{Y}_{n}}) \] is an isomorphism; thus $\labelcref{0024-eqn-06}$ is fully faithful. We have a commutative diagram \begin{equation*} \begin{tikzpicture}[>=stealth, baseline=(current bounding box.center)] 
\matrix[matrix of math nodes,row sep=1.7em, column sep=1.7em, text height=1.7ex, text depth=0.5ex] { 
|[name=11]| \varprojlim_{n \in \N}\on{Vect}_{r}(Y_{n}^{2}) & |[name=12]| \varinjlim_{Y \subset U \subset X} \on{Vect}_{r}(X^{2}|_{U}) \\
|[name=21]| \varprojlim_{n \in \N}\on{Vect}_{r}(Y_{n}^{1}) & |[name=22]| \varinjlim_{Y \subset U \subset X} \on{Vect}_{r}(X^{1}|_{U}) \\
|[name=31]| \varprojlim_{n \in \N}\on{Vect}_{r}(Y_{n}^{0}) & |[name=32]| \varinjlim_{Y \subset U \subset X} \on{Vect}_{r}(X^{0}|_{U}) \\
|[name=41]| \varprojlim_{n \in \N}\on{Vect}_{r}(\mc{Y}_{n}) & |[name=42]| \varinjlim_{Y \subset U \subset X} \on{Vect}_{r}(\mc{X}|_{U}) \\
}; 
\draw[->,font=\scriptsize]
(12) edge node[above=-1pt] {$\phi^{2}$} (11)
(22) edge node[above=-1pt] {$\phi^{1}$} (21)
(32) edge node[above=-1pt] {$\phi^{0}$} (31)
(42) edge node[above=-1pt] {\labelcref{0024-eqn-05}} (41)
(41) edge (31) (42) edge (32);
\draw[->,font=\scriptsize,transform canvas={xshift= 3pt}](21) edge (11) (22) edge (12);
\draw[->,font=\scriptsize,transform canvas={xshift= 0pt}](21) edge (11) (22) edge (12);
\draw[->,font=\scriptsize,transform canvas={xshift=-3pt}](21) edge (11) (22) edge (12);
\draw[->,font=\scriptsize,transform canvas={xshift= 1.5pt}](31) edge (21) (32) edge (22);
\draw[->,font=\scriptsize,transform canvas={xshift=-1.5pt}](31) edge (21) (32) edge (22); \end{tikzpicture} \end{equation*} where the vertical diagrams induce equivalences to the appropriate categories of descent data. Since $X^{0},X^{1},X^{2}$ are Cohen-Macaulay of dimension $\ge 3$, by \cite[0EL7]{SP} we have that $\phi^{0},\phi^{1},\phi^{2}$ are equivalences of categories, hence \labelcref{0024-eqn-05} is an equivalence of categories.

(3): We have an exact sequence \[ \H^{i}({\mc{Y}_{n}},(I^{n}/I^{n+1})|_{\mc{Y}_{n}}) \to \H^{i}(\mc{Y}_{n},\G_{m}) \to \H^{i}(\mc{Y}_{n-1},\G_{m}) \to \H^{i+1}({\mc{Y}_{n}},(I^{n}/I^{n+1})|_{\mc{Y}_{n}}) \] for all $i \ge 0$ and $n \ge 0$. Since $\mc{X}$ is tame, by \cite[8.1]{ALPER-GMSFAS} we have that $\pi$ is a good moduli space morphism, so by \cite[4.7(i)]{ALPER-GMSFAS} the pullback $\pi|_{Y_{n}} : \mc{Y}_{n} \to Y_{n}$ is also a good moduli space morphism for any $n \ge 0$. By \Cref{20220120-25}, the pullback morphism \[ (\pi|_{Y_{n}})^{\ast} : \H^{i}(Y_{n},I^{n}/I^{n+1}) \to \H^{i}({\mc{Y}_{n}},(I^{n}/I^{n+1})|_{\mc{Y}_{n}}) \] is an isomorphism for all $i \ge 0$. By (iii), we have $\H^{i}(X,I^{n}/I^{n+1}) = 0$ for all $n \ge 1$ and $0 \le i \le 2$. In particular, the restrictions \begin{equation} \label{0005-eqn01} \H^{i}(\mc{Y}_{n},\G_{m}) \to \H^{i}(\mc{Y}_{n-1},\G_{m}) \end{equation} are isomorphisms for $i=0,1$ and $n \ge 1$. Taking $\mc{E} = \mc{O}_{\mc{X}}$ in (1), we have that $\xi_{\mc{X},0}$ is an isomorphism.

We take $r=1$ in \labelcref{0024-eqn-05}. Since $X$ is a proper $k$-scheme, if $U \subseteq X$ is any open subset containing $Y$, then it contains all the codimension 1 points of $X$. Hence, since $X$ is regular, the restriction $\on{Vect}_{1}(\mc{X}) \to \varinjlim_{Y \subset U \subset X} \on{Vect}_{1}(\mc{X}|_{U})$ is an equivalence of categories. The projection $\varprojlim_{n \in \N} \on{Vect}_{1}(\mc{Y}_{n}) \to \on{Vect}_{1}(\mc{Y}_{0})$ is essentially surjective since the maps \labelcref{0005-eqn01} are isomorphisms for $i=1$ and all $n \ge 1$; thus $\xi_{\mc{X},1}$ is surjective. 

We prove that $\xi_{\mc{X},1}$ is injective. Let $L \in \Pic(\mc{X})$ be a line bundle. For any $n \ge 0$, we have a spectral sequence \[ \mr{E}_{1}^{p,q} = \H^{q}(X^{p},(I^{n}|_{\mc{X}} \otimes_{\mc{O}_{\mc{X}}} L)|_{X^{p}}) \implies \H^{p+q}(\mc{X},I^{n}|_{\mc{X}} \otimes_{\mc{O}_{\mc{X}}} L) \] associated to the covering $X^{0} \to \mc{X}$. For any $p \ge 0$, we have that the projection maps $X^{p} \to X^{0}$ are finite, hence $I|_{X^{p}}$ is an anti-ample invertible $\mc{O}_{X^{p}}$-module. By \cite[III, 7.6]{HARTSHORNE}, there exists some $n_{L} \gg 0$ such that $\mr{E}_{1}^{p,q} = 0$ for $n \ge n_{L}$ and $0 \le p+q \le 2$. Thus, for $n \ge n_{L}$ we have \[\H^{q}(\mc{X},I^{n}|_{\mc{X}} \otimes_{\mc{O}_{\mc{X}}} L) = 0 \] for all $0 \le q \le 2$, hence \[ \H^{q}(\mc{X},I^{n}/I^{n+1}|_{\mc{X}} \otimes_{\mc{O}_{\mc{X}}} L) = 0 \] for all $0 \le q \le 1$ and the restrictions \begin{equation} \label{0005-eqn03} \Gamma(\mc{Y}_{n+1},L|_{\mc{Y}_{n+1}}) \to \Gamma(\mc{Y}_{n},L|_{\mc{Y}_{n}}) \end{equation} are isomorphisms. If $L|_{\mc{X}|_{Y}}$ is a trivial $\mc{O}_{\mc{X}|_{Y}}$-module, then $L|_{\mc{Y}_{n}}$ and $L^{\vee}|_{\mc{Y}_{n}}$ are trivial $\mc{O}_{\mc{Y}_{n}}$-modules for all $n$ by \labelcref{0005-eqn01}. Choose some $n' \ge \max(n_{L},n_{L^{\vee}})$ and choose sections $s_{n'} \in \Gamma(\mc{Y}_{n'},L|_{\mc{Y}_{n'}})$ and $s^{\vee}_{n'} \in \Gamma(\mc{Y}_{n'},L^{\vee}|_{\mc{Y}_{n'}})$ such that the associated $\mc{O}_{\mc{Y}_{n'}}$-module morphisms $\mc{O}_{\mc{Y}_{n'}} \to L|_{\mc{Y}_{n'}}$ and $L|_{\mc{Y}_{n'}} \to \mc{O}_{\mc{Y}_{n'}}$ are mutually inverse. By \labelcref{0005-eqn03}, for each $n \ge n'$ there exist unique lifts $s_{n+1} \in \Gamma(\mc{Y}_{n+1},L|_{\mc{Y}_{n+1}})$ and $s^{\vee}_{n+1} \in \Gamma(\mc{Y}_{n+1},L^{\vee}|_{\mc{Y}_{n+1}})$ of $s_{n}$ and $s^{\vee}_{n+1}$ respectively. Since \labelcref{0024-eqn-06} is fully faithful, there exist unique sections $s \in \Gamma(\mc{X},L)$ and $s^{\vee} \in \Gamma(\mc{X},L^{\vee})$ which restrict to the systems $\{s_{n}\}_{n \ge n'}$ and $\{s^{\vee}_{n}\}_{n \ge n'}$ and the corresponding $\mc{O}_{\mc{X}}$-module morphisms $\mc{O}_{\mc{X}} \to L$ and $L \to \mc{O}_{\mc{X}}$ are mutually inverse, hence $L$ itself is trivial.

(4): Let $\ell$ be a positive integer such that $\ell \in k^{\times}$ and let $\alpha' \in \H_{\et}^{2}(\mc{X},\G_{m})[\ell]$ be an $\ell$-torsion Brauer class such that $\alpha'|_{\mc{X}|_{Y}} = 0$ in $\H_{\et}^{2}(\mc{X}|_{Y},\G_{m})$. Let $\mc{G}' \to \mc{X}$ be a $\G_{m}$-gerbe corresponding to $\alpha'$; then $\mc{G}' \times_{\mc{X}} \mc{X}|_{Y} \to \mc{X}|_{Y}$ is a trivial $\G_{m}$-gerbe. Let $\mc{G} \to \mc{X}$ be a $\m_{\ell}$-gerbe whose class $[\mc{G}] \in \H_{\et}^{2}(\mc{X},\m_{\ell})$ lifts $\alpha' = [\mc{G}'] \in \H_{\et}^{2}(\mc{X},\G_{m})$. Since $[\mc{G}|_{Y}] \in \H_{\et}^{2}(\mc{X}|_{Y},\m_{\ell})$ has trivial image in $\H_{\et}^{2}(\mc{X}|_{Y},\G_{m})$, there exists some line bundle $L_{Y} \in \H_{\et}^{1}(\mc{X}|_{Y},\G_{m})$ which maps to $[\mc{G}|_{Y}]$ under the coboundary map $\partial : \H_{\et}^{1}(\mc{X}|_{Y},\G_{m}) \to \H_{\et}^{2}(\mc{X}|_{Y},\m_{\ell})$ of the Kummer sequence. By (3) applied to $\mc{X}$, there exists some line bundle $L_{X} \in \H_{\et}^{1}(\mc{X},\G_{m})$ restricting to $L_{Y}$; after replacing $[\mc{G}]$ by $[\mc{G}] - \partial([L_{X}])$, we may assume that $\mc{G}|_{Y} \to \mc{X}|_{Y}$ is a trivial $\m_{\ell}$-gerbe. Let $\mc{L}_{Y}$ be a 1-twisted invertible $\mc{O}_{\mc{G}|_{Y}}$-module. Since the map $\mc{G} \to \mc{X}$ is a $\m_{\ell}$-gerbe, the stack $\mc{G}$ is tame and the composition $\mc{G} \to \mc{X} \to X$ is a flat coarse moduli space morphism (i.e. satisfies (ii)). Thus, by (3) applied to $\mc{G}$, there exists a 1-twisted invertible $\mc{O}_{\mc{G}}$-module $\mc{L}_{X}$ such that $\mc{L}_{X}|_{\mc{G}|_{Y}} \simeq \mc{L}_{Y}$. Hence $\mc{G} \to \mc{X}$ is a trivial $\m_{\ell}$-gerbe, so $\mc{G}' \to \mc{X}$ is a trivial $\G_{m}$-gerbe as well. \end{proof}

\begin{lemma} \label{20220120-25} Let $\pi : \mc{X} \to X$ be a good moduli space morphism. For any quasi-coherent $\mc{O}_{X}$-module $\mc{F}$, the pullback \[ \pi^{\ast} : \H^{i}(X,\mc{F}) \to \H^{i}(\mc{X},\mc{F}|_{\mc{X}}) \] is an isomorphism for all $i \ge 0$. \end{lemma} \begin{proof} This follows from the Leray spectral sequence \[ \mr{E}_{2}^{p,q} = \H^{p}(X,\mb{R}^{q}\pi_{\ast}(\mc{F}|_{\mc{X}})) \Rightarrow \H^{p+q}(\mc{X},\mc{F}|_{\mc{X}}) \] where we have $\mb{R}^{i}\pi_{\ast}(\mc{F}|_{\mc{X}}) = 0$ for $i \ge 1$ by \cite[3.10(v)]{ALPER-GMSFAS}. \end{proof}

\section{Proofs of the main theorems}\label{sec-025}

In this section, we use a series of reductions and purity to reduce the main results to the \Cref{0005}, the Grothendieck--Lefschetz theorem for stacks.

\begin{pg} \label{0009} We note that in \Cref{0004} we are free to replace $S$ by its reduction $S_{\mr{red}}$ or by any scheme that is birational to it. \end{pg}

\begin{pg}[Reduction to $\dim S = 1$] \label{0008} By \Cref{0009}, we may replace $S$ by an affine open neighborhood of the reduced scheme $\overline{\{s\}}$ to assume that $S$ is affine and integral and $\dim S = d$. Let $x_{1},\dotsc,x_{d} \in \kappa(\eta_{S})$ be a transcendence basis of the function field $\kappa(\eta_{S})$ over $F$; after replacing $S$ by an open subscheme, we may assume that $x_{i} \in \Gamma(S,\mc{O}_{S})$ for all $i$; let $S \to \A_{F}^{d} = \Spec F[t_{1},\dotsc,t_{d}]$ be the quasi-finite dominant $F$-morphism sending $t_{i} \mapsto x_{i}$ for all $i$. Set $F' := F(t_{1},\dotsc,t_{d-1})$ and let $S' := S \times_{\A_{F}^{d}} \A_{F'}^{1}$ where $\A_{F'}^{1} \to \A_{F}^{d}$ corresponds to the natural map $F[t_{1},\dotsc,t_{d}] \to F'[t_{d}]$. Let $f : S' \to S$ be the projection. \begin{center} \begin{tikzpicture}[>=angle 90] 
\matrix[matrix of math nodes,row sep=2em, column sep=2em, text height=1.7ex, text depth=0.5ex] { 
|[name=11]| S' & |[name=12]| S \\ 
|[name=21]| \A_{F'}^{1} & |[name=22]| \A_{F}^{d} \\
}; 
\draw[->,font=\scriptsize]
(11) edge node[above=-1pt] {$f$} (12) (21) edge (22) (11) edge (21) (12) edge (22); \end{tikzpicture} \end{center} We note that $f((S')_{(0)}) \subseteq S_{(d-1)}$ and that $f : S' \to S$ induces an isomorphism of function fields $\kappa(\eta_{S}) \to \kappa(\eta_{S'})$; moreover $\trdeg_{k}(F') = \trdeg_{k}(F)+d-1 \ge 3$. Thus, by replacing $F,S,\alpha$ by $F',S',\alpha|_{S'}$, we may assume that $\dim S = 1$. \end{pg}

\begin{pg}[Reduction to a model over $k$] By \Cref{0008}, we may assume that $S$ is affine, integral, of finite type over $F$, and $\dim S = 1$. By a limit argument, we can choose a $k$-subalgebra $A \subseteq F$ such that \begin{enumerate}[(i)] \item $A$ is of finite type over $k$, \item there exists a finite type $A$-scheme $U$ and an $F$-isomorphism $U \times_{\Spec A} \Spec F \simeq S$, and \item $\alpha$ is in the image of $\Br(U) \to \Br(S)$. \end{enumerate}

After possibly replacing $U$ by an open subscheme, we may assume that $U$ is regular; since $k$ is perfect, we have that $U$ is smooth over $k$. After a further localization, we may assume that there exists a projective $k$-scheme $X$ of dimension $\dim X = \trdeg_{k}(\kappa(S))$ such that there exists an open immersion $U \to X$.

Let $\overline{S}$ be a projective closure of $S$; we may replace $\overline{S}$ by its normalization to assume that $\overline{S}$ is regular. Applying \cite[0BX7]{SP} to the composition $S \to U \to X$, we obtain a nonconstant $k$-morphism $g : \overline{S} \to X$. If $s \in S_{(0)}$ is a closed point, then $g(s)$ is a codimension 1 point of $X$; by \Cref{0006}, for all but finitely many ample divisors $Y \subset X$, we have $\eta_{Y} \in g(S_{(0)})$. Thus we reduce to proving \Cref{0007} below. \end{pg}

\begin{lemma} \label{0006} Let $F$ be a field, let $S$ be a proper $F$-scheme of dimension 1, let $X$ be a scheme, let $\mc{L}$ be a line bundle on $X$, let $f : S \to X$ be a nonconstant morphism. For any section $s \in \Gamma(X,\mc{L})$ such that $X_{s}$ is affine, the pullback $f^{\ast}s \in \Gamma(S,f^{\ast}\mc{L})$ vanishes at a closed point of $S$. \end{lemma}

\begin{proof} We may assume that $S$ is connected. We have that $f^{-1}(X_{s}) = S_{f^{\ast}s}$. Since $X_{s}$ is affine and $S$ is proper over $F$ and $\dim S \ge 1$, we have $f(S) \not\subseteq X_{s}$, hence $S_{f^{\ast}s} = f^{-1}(X_{s}) \ne S$. \end{proof}

\begin{theorem} \label{0007} Let $k$ be a perfect field, let $U$ be a smooth $k$-scheme of dimension $\dim U \ge 4$, and let $G \in \Br(U)$ be a finite subgroup. Assume that $\ell := |G|$ is invertible in $k$. Then there exist infinitely many points $u \in U^{(1)}$ such that the composition $G \to \Br(U) \to \Br(\kappa(u))$ is injective. \end{theorem}

\begin{proof} Let $X$ be a proper $k$-scheme such that $U$ admits an open embedding $U \subseteq X$; let $D := X \setminus U$ be the complement. By Temkin's improvement \cite[4.3.1(iii)]{TEMKIN-TDADBPA2017} of Gabber's theorem \cite[1.3]{ILLUSIE-OGRU2009}, there exists a smooth projective $k$-scheme $X'$ and a prime-to-$\ell$ alteration $f : X' \to X$ such that $D' := f^{-1}(D)$ is a strict normal crossings divisor in $X'$. Since $\gcd(\ell,[\kappa(\eta_{X'}):\kappa(\eta_{X})]) = 1$, if $\alpha \in G$ is an $\ell$-torsion Brauer class such that $f^{\ast}\alpha = 0$ in $\Br(X')$, then $\alpha = 0$ itself. Hence we may replace $X,G$ by $X',f^{\ast}G$ to assume that $X$ is smooth projective over $k$ and that $D \subset X$ is a strict normal crossings divisor.

Let $D_{1},\dotsc,D_{r}$ be the irreducible components of $D$ and let \[ \mc{X} := \sqrt[\ell]{(X,D_{1})} \times_{X} \dotsb \times_{X} \sqrt[\ell]{(X,D_{r})} \] be the product of the $\ell$th root stacks. Since $D$ is strict normal crossings in $X$, we have that $\mc{X}$ is regular. By the argument of \cite[3.2.1]{LIEBLICH-PAIITBGOAAS}, there exists an open substack $\mc{U}' \subset \mc{X}$ containing $U \times_{X} \mc{X} \simeq U$ and all codimension 1 points of $\mc{X}$ and such that $G$ is contained in the subgroup $\Br(\mc{U}') \subseteq \Br(U)$. By purity for the Brauer group \cite{FUJIWARA-APOTAPC}, the restriction $\Br(\mc{X}) \to \Br(\mc{U}')$ is an isomorphism, so we may view $G$ as a subgroup of $\Br(\mc{X})$.

Let $\mc{O}_{X}(1)$ be an ample line bundle on $X$. By \cite[III, 7.6]{HARTSHORNE}, we may choose $N \gg 0$ such that $\H^{i}(X,\mc{O}_{X}(-nN)) = 0$ for all $0 \le i \le 3$ and $n \ge 1$. By Bertini's theorem, we may choose infinitely many smooth ample divisors $Y \in |\mc{O}_{X}(N)|$ such that $D|_{Y} \subset Y$ is a strict normal crossings divisor; then the restriction \[ \mc{X}|_{Y} \simeq \sqrt[\ell]{(Y,D_{1}|_{Y})} \times_{Y} \dotsb \times_{Y} \sqrt[\ell]{(Y,D_{r}|_{Y})} \] is regular. By \Cref{0005}(4), the restriction $\Br(\mc{X}) \to \Br(\mc{X}|_{Y})$ is injective. Since $\mc{X}|_{Y}$ is regular, the restriction $\Br(\mc{X}|_{Y}) \to \Br(\kappa(\eta_{Y}))$ is injective. \end{proof}

\begin{corollary} \label{0016} Let $k$ be a perfect field, let $S$ be a smooth $k$-scheme of dimension $\dim S \ge 4$, let $\alpha \in \Br(S)$ be a Brauer class such that $\per(\alpha)$ is invertible in $k$. If $\alpha|_{s} = 0$ for all codimension 1 points $s \in S^{(1)}$, then $\alpha = 0$. \end{corollary} \begin{proof} This follows from \Cref{0007}. \end{proof}

\begin{remark} \label{0023} In \Cref{0007}, suppose that the subgroup $G$ is unramified (i.e. $D = \emptyset$, so $G \subseteq \Br(X)$). Using classical results from SGA 2, we may give a short proof that, for a smooth projective $k$-scheme $X$ of dimension $\dim X \ge 4$ and any smooth ample divisor $Y \subset X$, the $\ell$-primary component of $\ker(\Br(X) \to \Br(Y))$ is $\ell$-divisible. Let $\alpha \in \Br(X)$ be a Brauer class and set $\ell := \per(\alpha)$. After replacing $k$ by a prime-to-$\ell$ extension, we may assume that $k$ contains a primitive $\ell$th root of unity; in this case we may choose an isomorphism $\Z/(\ell) \simeq \m_{\ell}$ of sheaves on the big fppf site of $\Spec k$. The Kummer sequence induces a commutative diagram \begin{equation} \label{0023-eqn01} \begin{tikzpicture}[>=angle 90, baseline=(current bounding box.center)] 
\matrix[matrix of math nodes,row sep=2em, column sep=2em, text height=1.7ex, text depth=0.5ex] { 
|[name=11]| \H_{\et}^{2}(X,\m_{\ell}) & |[name=12]| \H_{\et}^{2}(X,\G_{m}) & |[name=13]| \H_{\et}^{2}(X,\G_{m}) & |[name=14]| \H_{\et}^{3}(X,\m_{\ell}) \\ 
|[name=21]| \H_{\et}^{2}(X,\m_{\ell}) & |[name=22]| \H_{\et}^{2}(Y,\G_{m}) & |[name=23]| \H_{\et}^{2}(Y,\G_{m}) & |[name=24]| \H_{\et}^{3}(Y,\m_{\ell}) \\ 
}; 
\draw[->,font=\scriptsize]
(11) edge (12) (12) edge node[above=-1pt] {$\times \ell$} (13) (13) edge (14) (21) edge (22) (22) edge node[below=-1pt] {$\times \ell$} (23) (23) edge (24) (11) edge node[left=-1pt] {$f_{2}$} (21) (12) edge (22) (13) edge (23) (14) edge node[left=-1pt] {$f_{3}$} (24); \end{tikzpicture} \end{equation} where the rows are exact and the vertical arrows are restriction maps. Since $\dim X \ge 4$, taking $i=3,n=\dim X,c = 0$ in \cite[Exp. XIV, 5.7]{SGA2new} gives that $f_{2}$ is bijective and $f_{3}$ is injective. The claim then follows from a diagram chase on \labelcref{0023-eqn01}. \end{remark}

\begin{remark} \label{0002} We note that, in \Cref{0007}, it is not necessarily true that the set of points $u \in U$ such that $G \to \Br(U) \to \Br(\kappa(u))$ is not injective is finite. One way to construct examples of Brauer classes $\alpha \in \Br(S)$ vanishing at infinitely many closed points is to arrange that there exists a surjective $k$-morphism $f : S' \to S$ where $S'(k)$ is infinite and $f^{\ast}\alpha = 0$; if so, then $f(S'(k)) \cap S[\alpha^{-1}] = \emptyset$ by \Cref{0013}. For example, let $k$ be an infinite field of characteristic $\ch k \ne 2$, let $S := \P_{k}^{1} \setminus \{0,\infty\}$, let $a \in k^{\times} \setminus (k^{\times})^{2}$ be a non-square constant, let $\mc{A} := (a,t)$ be the quaternion algebra on $S$. Then $\alpha := [\mc{A}] \in \Br(S)$ is nontrivial \cite[1.3.8]{GILLE-SZAMUELY}, but $\mc{A}$ is trivialized after pullback by the squaring map $f : \P_{k}^{1} \to \P_{k}^{1}$, hence $f(S(k)) \cap S[\alpha^{-1}] = \emptyset$, i.e. for any square constant $b \in (k^{\times})^{2}$, the specialization $[\mc{A}|_{b}] \in \Br(k)$ is trivial. \par In general, by results of Yanchevskii \cite{YANCHEVSKII-KUOCBASFOSCA1985} and Mestre \cite{MESTRE-APCDVDEDBR2KXAQP1994}, \cite{MESTRE-APCDVDEDBR2KXAHPARC1994}, \cite{MESTRE-APCDVDEDBR2KXOKEUCF1996}, \cite{MESTRE-APCDVDEDBR2KXACP1996}, it is known that certain Brauer classes on open subschemes of $\P_{k}^{1}$ are trivialized after pullback by finite morphisms $\P_{k}^{1} \to \P_{k}^{1}$. More precisely, let $S \subseteq \P_{k}^{1}$ be an open subscheme and let $\alpha \in \Br(S)$ be a Brauer class such that $S(k) \setminus S[\alpha^{-1}] \ne \emptyset$. If either \begin{enumerate}[(i)] \item $k$ is Henselian, or \item $\alpha$ is $2$-torsion and $\sum_{x \in \P_{k}^{1} \setminus S} [\kappa(x):k] \le 4$, \end{enumerate} then there exists a finite morphism $f : \P_{k}^{1} \to \P_{k}^{1}$ such that $f^{\ast}\alpha = 0$ (see also \cite[II.Appendix]{SERRE-GALOISCOHOMOLOGY}). \end{remark}

\begin{lemma} \label{0013} Let $f : X \to Y$ be a morphism of schemes, let $\alpha \in \Br(Y)$ be a Brauer class such that $f^{\ast}\alpha \in \Br(X)$ is trivial, let $y \in Y$ be a point such that the fiber $X_{y} := X \times_{Y} \Spec \kappa(y)$ admits a $\kappa(y)$-point. Then $y \not\in Y[\alpha^{-1}]$. \end{lemma} \begin{proof} Since $(f^{\ast}\alpha)|_{X_{y}} = 0$, we have that $\alpha|_{y} = 0$. \end{proof}

\begin{pg}[{Proof of \Cref{0022}}] Let $U \subseteq S$ be an open subset and let $\alpha \in \Br(U)$ be a Brauer class such that the restriction $\alpha|_{K}$ is contained in $\Sh_{\Omega_{S}} \Br(K)$. For every codimension 1 point $s \in U$, we have $\alpha|_{K_{s}} = 0$, so $\alpha|_{\mc{O}_{S,s}^{\wedge}} = 0$ since $\mc{O}_{S,s}^{\wedge}$ is regular; thus $\alpha|_{\kappa(s)} = 0$ since $\mc{O}_{S,s}^{\wedge}$ is henselian. By \Cref{0004}, we have $\alpha = 0$. \end{pg}

\begin{question} Does \Cref{0004} still hold for fields $F$ such that $\trdeg_{k}(F) = 2$? \end{question} \begin{spg} For this, we would need to generalize the Noether-Lefschetz theorem \cite{GRIFFITHSHARRIS-OTNLTASROCTC1985} to separated Deligne-Mumford stacks of dimension $3$ satisfying the conditions of \Cref{0005}. For explicit evidence of an affirmative answer, see the example in \Cref{0021-03}. \end{spg} \begin{spg} As observed in the introduction, it is not enough to assume that $\trdeg_{k}(F) = 1$, for the following reason. Let $k$ be an algebraically closed field and let $F/k$ be a finitely generated field extension such that $\trdeg_{k}(F) = 1$. Let $S$ be a finite type $F$-scheme and let $\alpha \in \Br(S)$ be a Brauer class. For any closed points $s \in S$, we have that $\kappa(s)$ is a $C_{1}$-field, so $\alpha|_{s} = 0$ in $\Br(\kappa(s))$. \end{spg}

\begin{question} Does \Cref{0007} hold for $p$-torsion classes in characteristic $p$? \end{question}

\begin{spg} Let $F$ be a perfect field of characteristic $p$, let $S$ be a finite type $F$-scheme, let $\alpha \in \Br(S)$ be a $p$-torsion Brauer class. For any closed point $s \in S$, the specialization $\alpha|_{s} \in \Br(\kappa(s))$ is a $p$-torsion Brauer class over a perfect field of characteristic $p$, hence in fact $\alpha|_{s} = 0$. This gives one way to construct Brauer classes vanishing at all the closed points of a variety. However, this does not give a counterexample to \Cref{0004} since finitely generated fields of positive transcendence degree are not perfect. \end{spg}

\begin{question} \label{0021} Let $k$ be a perfect field, let $F/k$ be a finitely generated extension of transcendence degree $\trdeg_{k}(F) \ge 3$, let $S$ be a finite type $F$-scheme. Let $G\subset\Br(S)$ be a finite subgroup such that $|G|$ is invertible in $k$, and suppose $s \in S$ is a point such that the composition $G \to \Br(S) \to \Br(\kappa(s))$ is injective. For a given subgroup $G' \subseteq G$, does there exist an infinite set of closed points $s' \in \overline{\{s\}}$ such that the kernel of the composition \[ G \to \Br(S)\to\Br(\kappa(s')) \] is $G'$? \end{question}

\begin{spg} If $G' = 0$, the answer is ``yes'' by \Cref{0004}. \end{spg} \begin{spg} If $G' = G$, the answer is ``yes'' by \Cref{0020} below. A related question was considered by Frei, Hassett, Varilly-Alvarado in \cite{FREIHASSETTVARILLYALVARADO-ROBCOK3SRADE2021}. Namely, given a number field $k$ and a smooth projective $k$-scheme $X$ and a Brauer class $\alpha$, they define $\mc{S}(X,\alpha)$ to be the set of finite places $\mf{p}$ of $k$ such that $X$ has good reduction at $\mf{p}$, $\alpha$ is unramified at $\mf{p}$, and $\alpha|_{X_{\mf{p}}} = 0$ in $\Br(X_{\mf{p}})$. They prove that, if $X$ is a K3 surface such that the transcendental cohomology $T(X)$ satisfies a certain condition, then the set $\mc{S}(X,\alpha)$ has positive natural density. \end{spg}

\begin{spg} \label{0021-03} Here is an example where $G \simeq \Z/n\Z$ and $G' \simeq d\Z/n\Z$ for an arbitrary divisor $d$ of $n$. Let $F := \C(t_{1},t_{2})$ be the function field in two indeterminates, let $A := \Spec F[t_{3}^{\pm}]$ and $S := \Spec A$, and \[ \xi := t_{1}t_{2} \in F \] and let $F(\xi^{1/n})/F$ be the cyclic extension obtained by adjoining an $n$th root of $\xi$. Let $\chi \in \Aut(F(\xi^{1/n})/F)$ be a generator and let \[ \alpha := (S(\xi^{1/n})/S,\chi,t_{3}) \in \Br(S) \] denote the cyclic algebra of degree $n$ over $S$; we set $G := \langle \alpha \rangle$ and $G' := \langle \alpha^{n/d} \rangle$. 

Let $K/F$ be an extension. For any $K$-point $s : \Spec K \to S$ and any divisor $d$ of $n$, the restriction of the $d$th tensor power \[ s^{\ast}\alpha^{\otimes d} \simeq (K(\xi^{d/n})/K,\chi^{n/d},s^{\ast}t_{3}) \in \Br(K) \] is trivial if and only if the norm map \[ \mr{Nm}_{K(\xi^{d/n})/K} : K(\xi^{d/n})^{\times} \to K^{\times} \] contains $s^{\ast}t_{3} \in K^{\times}$ in its image \cite[\S15.1, Lemma]{PIERCE-AA1982}. Thus \Cref{0021} may be rephrased as follows: Does there exist infinitely many closed points $s \in S$ such that, if we denote the residue field of $s$ by $K := \kappa(s)$, the restriction $t_{3}|_{s} \in K^{\times}$ is in the image of $\mr{Nm}_{K(\xi^{d/n})/K}$ but not in the image of $\mr{Nm}_{K(\xi^{1/n})/K}$?

Taking $\{1,\xi^{1/m},\xi^{2/m},\dotsc,\xi^{(m-1)/m}\}$ as our $K$-basis of $K(\xi^{1/m})$, we have \[ \mr{Nm}_{K(\xi^{1/m})/K}(x_{0}1+x_{1}\xi^{1/m}+ \dotsb + x_{m-1}\xi^{(m-1)/m}) = \det M_{\xi,m} \] where $M_{\xi,m} \in \Mat_{m \times m}(F[x_{0},x_{1},\dotsc,x_{m-1}])$ is the matrix whose $(i,j)$th entry is \[ (M_{\xi,m})_{i,j} = \begin{cases} x_{i-j} & \text{if } i \ge j \\ \xi x_{i-j+m} &\text{if } i<j \end{cases} \] for all $1 \le i,j \le m$. For example, we have \begin{align*} &\mr{Nm}_{K(\xi^{1/4})/K}(x_{0}1+x_{1}\xi^{1/4}+x_{2}\xi^{2/4}+x_{3}\xi^{3/4}) \\ &= \det \begin{bmatrix}x_{0} & \xi x_{3} & \xi x_{2} & \xi x_{1} \\ x_{1} & x_{0} & \xi x_{3} & \xi x_{2} \\ x_{2} & x_{1} & x_{0} & \xi x_{3} \\ x_{3} & x_{2} & x_{1} & x_{0} \end{bmatrix} \\ &= (x_{0}^{2}-\xi x_{2}^{2})^{2}- \xi(x_{1}^{2}-\xi x_{3}^{2})^{2}+4\xi(x_{0}x_{1}-x_{2}x_{3}\xi)(x_{1}x_{2}+x_{0}x_{3}) \end{align*} and similarly \[ \mr{Nm}_{K(\xi^{1/2})/K}(y_{0}1+y_{1}\xi^{1/2}) = \det\begin{bmatrix} y_{0} & \xi y_{1} \\ y_{1} & y_{0} \end{bmatrix} = y_{0}^{2}-\xi y_{1}^{2} \] in the $m=4$ and $m=2$ cases, respectively.

Choose a polynomial $f \in \C[t_{1},t_{2}]$ such that $f|_{t_{1}=0} \in \C[t_{2}]$ is not a perfect $d$th power and let $s \in S(F)$ be the $F$-rational point corresponding to $t_{3} - f^{n/d}$. Then taking $(y_{0},y_{1}) = (f,0)$ shows that $t_{3}|_{s}$ is in the image of $\mr{Nm}_{F(\xi^{d/n})/F}$. If $t_{3}|_{s}$ is in the image of $\mr{Nm}_{F(\xi^{1/n})/F}$, then by clearing denominators we obtain polynomials $a_{0},a_{1},\dotsc,a_{n} \in \C[t_{1},t_{2}]$ such that \[ \det M_{\xi,n}(a_{0},a_{1},\dotsc,a_{n-1}) = f^{n/d}a_{n}^{n} \] in $\C[t_{1},t_{2}]$. Since $M_{\xi,n}(a_{0},a_{1},\dotsc,a_{n-1})|_{t_{1} = 0}$ is a lower-triangular matrix, the above simplifies to \[ (a_{0}|_{t_{1}=0})^{n} = (f|_{t_{1}=0})^{n/d}(a_{n}|_{t_{1}=0})^{n} \] in $\C[t_{2}]$ which is a contradiction since $(f|_{t_{1}=0})^{n/d}$ is not an $n$th power in $\C(t_{2})$. \end{spg}

\begin{theorem} \label{0020} Let $F$ be an infinite field, let $S$ be a finite type $F$-scheme of $\dim S \ge 1$, let $G \subseteq \Br(S)$ be a finite subgroup. Then there exist infinitely many codimension 1 points $s \in S^{(1)}$ such that $\alpha|_{s} = 0$ in $\Br(\kappa(s))$ for all $\alpha \in G$. \end{theorem}

\begin{proof} Let $\alpha_{1},\dotsc,\alpha_{n}$ be the elements of $G$, and for all $1\le i \le n$, let $Y_{i} \to S$ be a Brauer-Severi scheme corresponding to $\alpha_{i}$. Set \[ Y := Y_{1} \times_{S} \dotsb \times_{S} Y_{n} \] and let $f : Y \to S$ be the structure morphism. By \Cref{0018}, there exist infinitely many $s \in S^{(1)}$ such that the fiber \[ f^{-1}(s) = Y \times_{S} \Spec \kappa(s) \to \Spec \kappa(s) \] admits a section; for such $s$, we have $\alpha_{i}|_{s} = 0$ for all $1 \le i \le n$ as required. \end{proof}

\begin{lemma} \label{0017} Let $k$ be a field, let $X,Y$ be finite type $k$-schemes, let $\pi_{X} : X \times_{k} Y \to X$ and $\pi_{Y} : X \times_{k} Y \to Y$ be the two projections. For any point $z \in X \times_{k} Y$ such that $\pi_{Y}(z)$ is a $k$-point of $Y$, the extension of residue fields $\kappa(\pi_{X}(z)) \to \kappa(z)$ is an isomorphism. \end{lemma} \begin{proof} The $k$-point $\pi_{Y}(z) : \Spec k \to Y$ induces a section $\sigma : X \to X \times_{k} Y$ of $\pi_{X}$ which sends $\pi_{X}(z) \mapsto z$. \end{proof}

\begin{lemma} \label{0018} Let $k$ be an infinite field, let $X,Y$ be integral, finite type $k$-schemes with $\dim X \ge \dim Y \ge 1$, and let $f : X \to Y$ be a generically smooth $k$-morphism. Then there exist infinitely many points $y \in Y$ of codimension 1 such that the fiber \[ f^{-1}(y) = X \times_{Y} \Spec \kappa(y) \to \Spec \kappa(y) \] admits a section. \end{lemma}

\begin{proof} Let $\eta_{X}\in X$ and $\eta_{Y} \in Y$ denote the generic points. If $\dim X > \dim Y$, by \cite[056U]{SP} we may choose a closed point $x \in f^{-1}(\eta_{Y})$ such that $\kappa(x)/\kappa(\eta_{Y})$ is finite separable. By replacing $X$ by the reduced scheme $\overline{\{x\}}$, we may assume that $\dim X = \dim Y$ and that the function field extension $\kappa(\eta_{X})/\kappa(\eta_{Y})$ is a finite separable extension.

We may replace $X$ and $Y$ by open subschemes so that $X$ and $Y$ are affine. Set $A := \Gamma(X,\mc{O}_{X})$ and $B := \Gamma(Y,\mc{O}_{Y})$; after a further localization, we may assume (by the primitive element theorem) that $A = B[a]$ for some $a \in A$ which is integral over $B$. Let $\phi \in B[t]$ be the minimal polynomial of $a$ over $B$. If $\deg \phi = 1$, then every fiber of $f$ is trivial, so we may assume $\deg \phi \ge 2$. Let $k' \subseteq A$ be the integral closure of $k$ in $A$; then $k'/k$ is a finite extension by \cite[3.3.2]{HUNEKESWANSON-ICOIRAM2006}. Since $\dim(B) \ge 1$, we have $\dim_{k}(B) = \infty$ so we may choose some $b \in B$ such that $a+b \not\in k'$. By replacing $a$ by $a+b$, we may assume that $a \not\in k'$, i.e. $\phi$ has a root which is not integral over $k$, so $\phi$ does not divide any nonzero element of $k[t]$. This implies the composition \[ k[t] \to B[t] \to B[t]/(\phi) \simeq A \] is injective; hence the corresponding morphism \[ g : X \to \A_{k}^{1} \times_{k} Y \to \A_{k}^{1} \] is flat. Thus the set-theoretic image $U := g(X) \subseteq \A_{k}^{1}$ is an open subset of $\A_{k}^{1}$. Since $k$ is infinite, the open subset $U$ contains infinitely many $k$-points. For any $x \in X$ such that $g(x)$ is a $k$-point of $U$, the extension of residue fields $\kappa(f(x)) \to \kappa(x)$ is an isomorphism by \Cref{0017}. \end{proof}

\section{Brauer classes vanishing at a prescribed set of points} \label{sec-03}

\Cref{0004} and \Cref{0007} show that too many vanishing specializations of a Brauer class kill it. In this section, we consider the problem of constructing non-zero Brauer classes that vanish at a given finite set of points (i.e. whether there always exist $\alpha \in \Br(S)$ such that $S[\alpha^{-1}]$ avoids an arbitrary set of points).

\begin{question} \label{0003} Let $k$ be a field, let $S$ be a curve over $k$. Let $T_{1},T_{2} \subseteq S$ be disjoint subsets of $S$. Does there exist a Brauer class $\alpha \in \Br(S)$ such that $T_{1} \subseteq S[\alpha^{-1}]$ and $T_{2} \cap S[\alpha^{-1}] = \emptyset$? \end{question}

\begin{spg} In \Cref{sec-02}, we are mostly interested in whether $S[\alpha^{-1}]$ is stable under specialization. We may ask whether $S[\alpha^{-1}]$ is stable under generalization as well. If $s \in S$ is a point such that the reduced subscheme $\overline{\{s\}} \subset S$ is regular, then $\alpha|_{s} = 0$ implies $\alpha|_{s'} = 0$ for all $s' \in \overline{\{s\}}$. On the other hand, if $\overline{\{s\}}$ is not regular, it is possible that $\alpha|_{s} = 0$ but $\alpha|_{s'} \ne 0$ for some $s' \in \overline{\{s\}}$: \par Set $A := (\Q[x,y]/(x^{2} - y^{3} + 2y^{2}))_{\langle x,y \rangle}$; then $A$ is a local Noetherian domain of dimension 1 and a normalization is given by the map $A \to \Q[t]_{\langle t \rangle}$ sending $x \mapsto t(t^{2}+2)$ and $y \mapsto t^{2}+2$. Set $S := \Spec A$ with generic point $s$ and closed point $s'$, and consider the quaternion algebra $\alpha := (y-1,y-1) \in \Br(S)$. We have $\alpha|_{s'} \ne 0$ since the conic $-X_{0}^{2}-X_{1}^{2}=X_{2}^{2}$ does not have a nontrivial solution over $\kappa(s') \simeq \Q$; however $\alpha|_{s} = 0$ since the conic $(t^{2}+1)X_{0}^{2}+(t^{2}+1)X_{1}^{2} = X_{2}^{2}$ has the solution $(X_{0},X_{1},X_{2}) = (t,1,t^{2}+1)$ over $\kappa(s) \simeq \Q(t)$. \end{spg}

\begin{spg} In \Cref{0011}, we give a positive answer to a weaker version of \Cref{0003} where we allow $\alpha$ to be ramified on $S$. In \Cref{0012}, we prove the existence of unramified Brauer classes for the case $|T| = 1$. \end{spg}

\begin{proposition} \label{0011} Let $k$ be a Hilbertian field, let $S$ be a smooth proper curve over $k$. For any finite set of closed points $T \subseteq S_{(0)}$ and any positive integer $n$, there exists an open subscheme $S'$ of $S$ such that $T \subseteq S'$ and a Brauer class $\alpha \in \Br(S')$ such that $\eta_{S'} \in S'[\alpha^{-1}]$ and $T \cap S'[\alpha^{-1}] = \emptyset$ and $\gcd(\per \alpha,n) = 1$. \end{proposition}

\begin{proof} Let $S_{1},S_{2}$ be smooth proper curves over $k$, let $f : S_{1} \to S_{2}$ be a finite $k$-morphism, let $T_{1} \subseteq (S_{1})_{(0)}$ be a set of closed points, and set $T_{2} := f(T_{1})$. Suppose $S_{2}' \subseteq S_{2}$ is an open subset containing $T_{2}$ and $\alpha_{2} \in \Br(S_{2}')$ is a Brauer class such that $\eta_{S_{2}'} \in S_{2}'[\alpha_{2}^{-1}]$ and $T_{2} \cap S_{2}'[\alpha_{2}^{-1}] = \emptyset$. Then for $S_{1}' := f^{-1}(S_{2}')$ and $\alpha_{1} := f^{\ast}\alpha_{2} \in \Br(S_{1}')$, we have $T_{1} \cap S_{1}'[\alpha_{1}^{-1}] = \emptyset$. In this setup, we have \[ \per \alpha_{1} | \per \alpha_{2} | (\deg f)\per \alpha_{1} \] so $\gcd(\per \alpha_{2},n \deg f) = 1$ implies $\per \alpha_{1} = \per \alpha_{2}$.

By the above, we may reduce to the case $S = \P_{k}^{1}$ by choosing a finite $k$-morphism $f : S \to \P_{k}^{1}$ and replacing $S,T,n$ by $\P_{k}^{1},f(T),n \deg f$, respectively.

By a translation of $\P_{k}^{1}$, we may assume that $0,\infty \not\in T$. For any $s \in T$, let $\xi_{s} \in k[t]$ be the monic irreducible polynomial defining the closed subscheme $\Spec \kappa(s) \to \P_{k}^{1}$ and set $\xi := \prod_{s \in T} \xi_{s}$. Choose a positive integer $m \ge 2$ such that $\gcd(m,n) = 1$. Since $K$ is Hilbertian, by \cite[16.3.6]{FRIED-JARDEN-FIELDARITHMETIC} there exists a cyclic Galois extension $k'/k$ of degree $[k':k] = m$. Choose a positive integer $i \in \Z_{\ge 1}$ such that $m \nmid i+\deg \xi$, let $\chi \in \Aut(k'/k)$ be a generator. We show that the Brauer class $\alpha \in \Br(k(t))$ corresponding to the cyclic algebra \[ \mc{A} := (k'(t)/k(t),\chi,t^{i}\xi + 1) \] has the desired properties. To show that $\mc{A}$ is nontrivial, it suffices by \cite[\S15.1, Lemma]{PIERCE-AA1982} to show that the unit $t^{i}\xi+1 \in (k(t))^{\times}$ is not equal to the norm of any element of $(k'(t))^{\times}$. For this, we note that for any $\frac{a}{b} \in (k'(t))^{\times}$, the norm $N_{k'(t)/k(t)}(\frac{a}{b}) = \prod_{\sigma \in \Aut(k'/k)} \sigma(\frac{a}{b})$ has degree $m (\deg(a) - \deg(b))$, but $\deg(t^{i}\xi + 1)$ is not a multiple of $m$ (by choice of $i$). For any closed point $s \in T$, we have that $\mc{A}$ is unramified at $s$ (since $\gcd(\xi_{s},t^{i}\xi+1) = 1$) and the restriction \[ \mc{A}|_{s} = (k' \otimes_{k} \kappa(s)/\kappa(s),\chi,(t^{i}\xi + 1)|_{\kappa(s)}) \simeq (k' \otimes_{k} \kappa(s)/\kappa(s),\chi,1) \] is a trivial central simple $\kappa(s)$-algebra. \end{proof}

\begin{proposition} \label{0012} Let $k$ be a field that is finitely generated over a global field, let $S$ be a smooth proper geometrically integral curve over $k$ of genus $g \ge 1$. For any closed point $x \in S_{(0)}$, there exists a Brauer class $\alpha \in \Br(S)$ such that $\eta_{S} \in S[\alpha^{-1}]$ and $x \not\in S[\alpha^{-1}]$.  \end{proposition}

\begin{proof} If $x$ is a $k$-point, then the Leray spectral sequence gives a split exact sequence \[ 0 \to \Br(k) \to \Br(S) \to \H^{1}(\Spec k , \Jac_{S}) \to 0 \] where $\Jac_{S}$ is the Jacobian of $S$. Since $k$ is Hilbertian, by \cite[Theorem 1.6]{CLARKLACY-TAGOCOEIOEIFGF2019} the group $\H^{1}(\Spec k , \Jac_{S})$ contains elements of order $n$ for any positive integer $n$; given any such element, we obtain a nontrivial element of $\Br(S)$ which vanishes at $x$.

If $x$ is not a $k$-point, then the extension $\kappa(x)/k$ is nontrivial; by a theorem of Fein, Schacher \cite{FEINSCHACHER-RELATIVEBRAUERGROUPS-III} the relative Brauer group $\Br(\kappa(x)/k)$ is infinite. For any nontrivial class $\alpha \in \Br(\kappa(x)/k)$, the constant class $\alpha|_{S} \in \Br(S)$ is a Brauer class which vanishes at $x$. \end{proof}

\begin{remark} In \Cref{0012}, the condition $S \not\simeq \P_{k}^{1}$ is necessary since the pullback $\Br(k) \to \Br(\P_{k}^{1})$ is an isomorphism (so if $x \in \P_{k}^{1}(k)$ is a $k$-point then there does not exist any nontrivial $\alpha \in \Br(\P_{k}^{1})$ which vanishes at $x$). \end{remark}

\section{Applications to rational points on genus 1 curves} \label{sec-04}

As an application of our results in \Cref{sec-02}, we prove \Cref{0015}, which may be viewed as a local-to-global principle for genus $1$ curves over function fields of fourfolds over $k$. We first prove a version of \cite[Lemma 3.2]{CIPERIANIKRASHEN-RBGOGOC2012} for relative elliptic curves.

\begin{lemma} \label{0014} Let $S$ be a quasi-compact scheme admitting an ample line bundle, let $\pi : \mc{E} \to S$ be a relative elliptic curve, and $\sigma : S \to \mc{E}$ denote the identity section of $\pi$. There is a natural isomorphism \[ \ker(\sigma^{\ast} : \Br(\mc{E}) \to \Br(S)) \simeq \H_{\et}^{1}(S,\mc{E})_{\mr{tors}} \] which is functorial on $S$. \end{lemma}
\begin{proof} We consider the Leray spectral sequence for $\pi$. Since $\pi$ is a relative curve, we have $\mb{R}^{2}\pi_{\ast}\G_{m} = 0$ by \cite[II, Lemma 2\ensuremath{'}]{GABBER-THESIS}. Since $\pi$ admits a section, the pullback $\pi^{\ast} : \H_{\et}^{3}(S,\G_{m}) \to \H_{\et}^{3}(\mc{E},\G_{m})$ is injective, hence the differential $d_{2}^{1,1} : \H_{\et}^{1}(S,\mb{R}^{1}\pi_{\ast}\G_{m}) \to \H_{\et}^{3}(S,\G_{m})$ is the zero map. Thus we obtain an exact sequence \begin{align} \label{04-01} 0 \to \H_{\et}^{2}(S,\G_{m}) \stackrel{\pi^{\ast}}{\to} \H_{\et}^{2}(\mc{E},\G_{m}) \to \H_{\et}^{1}(S,\mb{R}^{1}\pi_{\ast}\G_{m}) \to 0 \end{align} which admits a canonical splitting induced by $\sigma^{\ast}$. We have isomorphisms \[ \mb{R}^{1}\pi_{\ast}\G_{m} \simeq \Pic_{\mc{E}/S} \simeq \Pic_{\mc{E}/S}^{0} \times \underline{\Z} \simeq \mc{E} \times \underline{\Z} \] of etale sheaves on $S$. Since $\H_{\et}^{1}(S,\underline{\Z})$ is torsion-free (see e.g. \cite[A.3]{SHIN-TCBGOATGMG}), we have a split exact sequence \[ 0 \to \Br(S) \stackrel{\pi^{\ast}}{\to} \Br(\mc{E}) \to \H_{\et}^{1}(S,\mc{E})_{\mr{tors}} \to 0 \] obtained by restricting to the torsion subgroups in \labelcref{04-01}. \end{proof}

\begin{pg}[{Proof of \Cref{0015}}] By replacing $S$ with an open subscheme, we may assume that $S$ is affine and regular; furthermore, we may assume that $C$ has good reduction over $S$, i.e. that there exists a scheme $\mc{C}$ and a smooth projective morphism $\pi : \mc{C} \to S$ whose geometric fibers are connected curves of genus 1 and such that there is a $K$-isomorphism $\pi^{-1}(\eta_{S}) \simeq C$, where $\eta_{S} \in S$ denotes the generic point. Let $\mc{E} := \Pic_{\mc{C}/S}^{0}$ be the Jacobian of $\mc{C}$. We have that $\mc{C}$ is an $\mc{E}$-torsor, i.e. there is a class $[\mc{C}] \in \H_{\et}^{1}(S,\mc{E})$ which is the obstruction to the existence of a section of $\pi$. Let \[ \alpha \in \Br(\mc{E}) \] be the Brauer class corresponding to $[\mc{C}]$ under the isomorphism in \Cref{0014}.

We will show that $\alpha|_{x} = 0$ for all closed points $x \in \pi^{-1}(\eta_{S})$; if so, then $\alpha|_{\pi^{-1}(\eta_{S})} = 0$ by \Cref{0004} applied to $\pi^{-1}(\eta_{S})$. Let $\til S := \overline{\{x\}}$ be the closure of $x$ in $\mc{E}$, equipped with the reduced scheme structure. As the composition $\til S \to \mc{E} \to S$ is generically finite and dominant, it follows that $\til S$ is itself an integral finite type $F$-scheme of dimension $\dim \til S \ge 1$. By restricting $S$ again to a smaller Zariski open set, we can assume that $\til S \to S$ is finite flat and that $\til S$ is smooth over $F$. To show that $\alpha|_{x} = 0$, by \Cref{0004} applied to $\til S$, it suffices to show that $\alpha|_{\til s} = 0$ for every codimension $1$ point $\til s \in \til S$. Choose such a point $\til s \in \til S$ and consider its image $s \in S$, which is also codimension $1$ since $\til S \to S$ is finite flat. By hypothesis, the restriction $[\mc{C} \times_{S} \Spec K_{s}] \in \H_{\et}^{1}(\Spec K_{s}, \mc{E} \times_{S} \Spec K_{s})$ is trivial, so $\alpha|_{\mc{E} \times_{S} \Spec K_{s}} = 0$ by \Cref{0014}. Since $\mc{E} \times_{S} \Spec \mc{O}_{S,s}^{\wedge}$ is regular, we have $\alpha|_{\mc{E} \times_{S} \Spec \mc{O}_{S,s}^{\wedge}} = 0$, thus $\alpha|_{\mc{E} \times_{S} \Spec \kappa(s)} = 0$, in particular $\alpha|_{\til s} = 0$. \end{pg}

\bibliography{main.bib}
\bibliographystyle{alpha}

\end{document}